\documentclass[12pt]{article}

\usepackage{amsthm,amsmath,amssymb,amsfonts}

\usepackage{graphicx,tikz}

\usepackage[colorlinks=true,citecolor=black,linkcolor=black,urlcolor=blue]{hyperref}

\newcommand{\arxiv}[1]{\href{http://arxiv.org/abs/#1}{\texttt{arXiv:#1}}}

\theoremstyle{plain}
\newtheorem{theorem}{Theorem}[section]
\newtheorem{lemma}[theorem]{Lemma}
\newtheorem{corollary}[theorem]{Corollary}
\newtheorem{proposition}[theorem]{Proposition}

\theoremstyle{definition}
\newtheorem{definition}[theorem]{Definition}
\newtheorem{example}[theorem]{Example}

\theoremstyle{remark}
\newtheorem{remark}[theorem]{Remark}

\title{\bf Weighted quasisymmetric enumerator for generalized permutohedra}

\author{Vladimir Gruji\'c\\
\small Faculty of Mathematics\\[-0.8ex]
\small Belgrade University\\[-0.8ex]
\small\tt vgrujic@matf.bg.ac.rs\\
\and
Marko Pe\v{s}ovi\'c\\
\small Faculty of Civil Engineering\\[-0.8ex]
\small Belgrade University\\[-0.8ex]
\small\tt mpesovic@grf.bg.ac.rs \and
Tanja Stojadinovi\'c\\
\small Faculty of Mathematics\\[-0.8ex]
\small Belgrade University\\[-0.8ex]
\small\tt tanjas@matf.bg.ac.rs}

\date{
\small Mathematics Subject Classifications: 52B40, 52B05, 16T05}

\begin{document}

\maketitle

\begin{abstract}
We introduce a weighted quasisymmetric enumerator function
associated to generalized permutohedra. It refines the Billera,
Jia and Reiner quasisymmetric function which also includes the
Stanley chromatic symmetric function. Beside that it carries
information of face numbers of generalized permutohedra. We
consider more systematically the cases of nestohedra and matroid
base polytopes.

\bigskip\noindent \textbf{Keywords}: generalized permutohedron,
quasisymmetric function, matroid, matroid base polytope,
combinatorial Hopf algebra, f-polynomial

\end{abstract}

\section{Introduction}

For a generalized permutohedron $Q$ there is a quasisymmetric
enumerator function $F(Q)$ introduced by Billera, Jia and Reiner
in \cite{BJR}. It enumerates positive integer lattice points
$\omega=(\omega_1,\ldots,\omega_n)\in\mathbb{Z}_+^{n}$ which are
$Q$-generic. It means that the weight function
$\omega^{\star}:Q\rightarrow\mathbb{R}$ defined by
$\omega^{\star}(x)=\langle\omega,x\rangle$ has its maximum at a
unique vertex $v$ of $Q$. That is

\begin{equation}\label{bjr}
F(Q)=\sum_{\omega \ Q-{\rm generic}}\mathbf{x}_\omega,
\end{equation} where
$\mathbf{x}_\omega=x_{\omega_1}x_{\omega_2}\cdots x_{\omega_n}$.
The generalized permutohedra are introduced and extensively
studied by Postnikov \cite{P} and Postnikov, Reiner and Williams
\cite{PRW}. They are deformations of the standard permutohedra
obtained by moving facets in normal directions. The faces of the
standard permutohedron $Pe^{n-1}$ are labelled by flags
$\mathcal{F}$ of subsets of the set $[n]$. A part of the reach
combinatorial structure of a generalized permutohedron $Q$ is a
certain statistic $\mathrm{rk}_Q$ on the face lattice of the
standard permutohedron $Pe^{n-1}$ which we call the $Q$-{\it
rank}. The $Q$-rank of a face $\mathcal{F}$ of $Pe^{n-1}$ is the
dimension of the face of $Q$ in which $\mathcal{F}$ is deformed.
The normal fan of the standard permutohedron $\Sigma_{Pe^{n-1}}$
is the braid arrangement fan. The space $\mathbb{R}^{n}$ is
decomposed by braid cones $\sigma_\mathcal{F}$. Each positive
integer vector $\omega\in\mathbb{Z}_+^{n}$ determines a unique
flag $\mathcal{F}_\omega$ such that $\omega$ lies in the relative
interior of the braid cone $\sigma_{\mathcal{F}_\omega}$. We
define the following weighted enumerator function associated to a
generalized permutohedron $Q$

\begin{equation}\label{def}
F_q(Q)=\sum_{\omega\in\mathbb{Z}_+^{n}}q^{\mathrm{rk}_Q(\mathcal{F}_\omega)}\mathbf{x}_\omega.
\end{equation}
The $Q$-rank of a face $\mathcal{F}$ is zero if it is deformed
into a vertex of $Q$. Henceforth $F_q(Q)$ specializes at $q=0$ to
the Billera, Jia and Reiner quasisymmetric enumerator function
$F_0(Q)=F(Q)$.

On the other hand, when a class of generalized permutohedra is
specified we obtain the well known combinatorial enumerators. The
case of graphical zonotopes $Q=Z_\Gamma$ is studied in \cite{GV2}.
The quasisymmetric function $F_q(Z_\Gamma)$ is a $q$-refinement of
the Stanley chromatic symmetric function $X_\Gamma$ of graphs

$$F_0(Z_\Gamma)=X_\Gamma.$$

The case of nestohedra is studied in \cite{GV1} for $q=0$ and for
a $q$-analog in \cite{GS}. For a subclass of graph-associahedra
$Q=P_\Gamma$ the enumerator $F(P_\Gamma)$ produces a new
quasisymmetric invariant of a graph $\Gamma$ with a nice behavior.

The both cases of graphical zonotopes and nestohedra have in
common that corresponding enumerators $F(Q)$ coincide with
universal morphisms from certain combinatorial Hopf algebras to
quasisymmetric functions. In the case of graphical zonotopes it is
the chromatic Hopf algebra of graphs and in the case of nestohedra
it is the non-cocommutative Hopf algebra on building sets and its
Hopf subalgebra of graphs. Billera, Jia and Reiner applied their
enumerator $F(Q)$ in the case of matroid base polytopes $Q=P_M$.
The invariant $F(P_M)$ also comes from a Hopf algebra, in this
case of matroids which was firstly introduced by Schmitt
\cite{Sch}.

A remarkable and unifying approach to Hopf monoid structures
constructed on combinatorial objects that provide generalized
permutohedra have been developed in the recently published paper
by M. Aguiar and F. Ardila \cite{AA}. Combining with the
universality of quasisymmetric functions in the category of
combinatorial Hopf algebras \cite{ABS} shows the naturality of the
invariant $F_q(Q)$ in enumerative and algebraic combinatorics.

The paper is organized as follows. In section 2 we review the
necessary facts about combinatorics of standard and generalized
permutohedra. In section 3 we review the basic facts about
quasisymmetric functions. In section 4 we introduce the
quasisymmetric function $F_q(Q)$ and show that it contains the
information about $f$-vectors of generalized permutohedra. It may
be regarded as a far-reaching illumination of the Stanley
$(-1)$-color theorem for numbers of acyclic orientations of a
graph. The cases of nestohedra and graph-associahedra are
considered in sections 5 and 6. In the rest of the paper the case
of matroid base polytopes is considered more thoroughly. In
section 7 we review some basic facts about combinatorics of
matroid base polytopes and introduce the combinatorial Hopf
algebra of matroids and its $q$-analog. Some calculation for
uniform matroids is presented. Finally in section 8 some
properties of the weighted quasisymmetric enumerator function of
matroids are derived.

\section{Generalized permutohedra}

The symmetric group $S_n$ acts on the space $\mathbb{R}^{n}$ by
permuting the coordinates. Recall that a $(n-1)$-dimensional
\emph{permutohedron} $Pe^{n-1}$ is the convex hull of the orbit of
a point with increasing coordinates $x_1<\cdots<x_n$

$$Pe^{n-1}=\mathrm{Conv}\{(x_{\omega(1)},x_{\omega(2)}\ldots,x_{\omega(n)})\mid \omega\in
S_n\}.$$

The \emph{braid arrangement} is the arrangement of hyperplanes
$\{x_i=x_j\}_{1\leq i<j\leq n}$ in the space $\mathbb{R}^n$. The
regions of the braid arrangement, called Weyl chambers are
labelled by permutations $\omega\in S_n$
$$C_\omega:=\{x_{\omega(1)}\leq x_{\omega(2)}\leq\ldots\leq x_{\omega(n)}\}.$$
The corresponding {\it braid arrangement fan} is the normal fan
$\Sigma_{Pe^{n-1}}$ of the permutohedron $Pe^{n-1}$. The cones of
the braid arrangement fan are called \emph{braid cones}.

A flag of the length $|\mathcal{F}|=k$ on the set
$[n]=\{1,\ldots,n\}$ is a chain of subsets
$$\mathcal{F}:\emptyset=:F_0\subset F_1\subset\ldots\subset F_{k-1}\subset F_{k}\subset F_{k+1}:=[n].$$
The type of a flag $\mathcal{F}$ is the following $(k+1)$-tuple of
integers

\begin{equation}\label{type}
\mathrm{type}(\mathcal{F})=(|F_1|-|F_0|,|F_2|-|F_1|,\ldots,|F_{k+1}|-|F_{k}|).
\end{equation}

The set of flags is ordered by refinements. We write
$\mathcal{G}\preceq\mathcal{F}$ if $\mathcal{F}$ refines
$\mathcal{G}$. There is an obvious order reversing one-to-one
correspondence between the face lattice of the permutohedron
$Pe^{n-1}$ and the lattice of flags of $[n]$. With no abuse of
notation we denote a face of the permutohedron $Pe^{n-1}$ by the
corresponding flag $\mathcal{F}$. Then
$\mathcal{G}\preceq\mathcal{F}$ if and only if
$\mathcal{F}\subseteq\mathcal{G}$ as faces of $Pe^{n-1}$. By this
convention we have that
$\mathrm{dim}\mathcal{F}=n-|\mathcal{F}|-1$. For example, the
facets correspond to the flags $\emptyset\subset A\subset[n]$,
while the vertices correspond to maximal flags.

The braid cone $\sigma_{\mathcal{F}}$ at the face $\mathcal{F}$ is
determined by the coordinates relations

\begin{equation}\label{braidcone}
\left\{\begin{array}{cc} x_p=x_q \ \ \mbox{if} \ \ p,q\in
F_{i+1}\setminus F_{i} \ \ \mbox{for some} \ \ i=0,\ldots,k, \\
x_p\leq x_q \ \ \mbox{if} \ \ p\in F_{i}\setminus F_{i-1} \ \
\mbox{and} \ \ q\in F_{i+1}\setminus F_{i} \ \ \mbox{for some} \ \
i=1,\ldots,k.\end{array}\right.
\end{equation}
Note that $\dim(\sigma_{\mathcal{F}})=|\mathcal{F}|$ and the
relative interior $\sigma_{\mathcal{F}}^{\circ}$, given by strict
inequalities in the second condition above, is homeomorphic to
$\mathbb{R}^{|\mathcal{F}|}.$ Conversely, the flag $\mathcal{F}$
can be reconstructed from the braid cone $\sigma_{\mathcal{F}}$ by
setting $p\in F_{i+1}\setminus F_{i}$ and $q\in F_{j+1}\setminus
F_j$ for some $0\leq i<j\leq k$ whenever $x_p<x_q$ for all points
in the relative interior $\sigma_{\mathcal{F}}^{\circ}$.

\begin{definition}
A convex polytope $Q$ is an $(n-1)$-dimensional {\it generalized
permutohedron} if its normal fan $\Sigma_Q$ is coarser than the
braid arrangement fan $\Sigma_{Pe^{n-1}}$.
\end{definition}

Generalized permutohedra are equivalently characterized as
deformations of the standard permutohedron $Pe^{n-1}$ by moving
its vertices with keeping directions of edges. Thus any edge of a
generalized permutohedron $Q$ lies in the direction of some
$e_i-e_j$, where $e_i, i=1,\ldots,n$ are the standard basis
vectors in $\mathbb{R}^{n}$. For equivalent descriptions of
generalized permutohedra see \cite{PRW}.

\begin{definition}\label{lattice}
For an $(n-1)$-dimensional generalized permutohedron $Q$ there is
a map of face lattices
$$\pi_Q:L(Pe^{n-1})\rightarrow L(Q)$$
determined by $\pi_Q(\mathcal{F})=G$ if and only if the relative
interior of the braid cone $\sigma_{\mathcal{F}}^{\circ}$ is
contained in the relative interior $\sigma_G^{\circ}$ of the
normal cone at the face $G$ of $Q$.
\end{definition}

\begin{proposition}\label{euler}
$$\displaystyle\sum_{\mathcal{F}:\pi_Q(\mathcal{F})=G}(-1)^{|\mathcal{F}|}=(-1)^{n-\dim(G)-1}.$$
\end{proposition}
\begin{proof}
Consider the collection of flags
$\pi_Q^{-1}(G)=\{\mathcal{F}_1,\mathcal{F}_2,\ldots
\mathcal{F}_t\}$. We have
$\sigma_G^{\circ}=\bigcup_{i=1}^{t}\sigma_{\mathcal{F}_i}^{\circ}$
and $\sigma_G^{\circ}$ is homeomorphic to the $k$-dimensional open
cell, where $k=n-\mathrm{dim}G-1$. Let $f_i, i=0,1,\ldots,k$ be
the numbers of $i$-dimensional cones such that
$\sigma_\mathcal{F}^{\circ}\subset\sigma_G^{\circ}$. By
inclusion-exclusion principle we have

$$f_k-f_{k-1}+\cdots+(-1)^{k}f_0=1.$$
\end{proof}

\section{Quasisymmetric functions}

In this section we review the basic facts about combinatorial Hopf
algebras and quasisymmetric functions. The notion of combinatorial
Hopf algebra, originated in the work of Aguiar, Bergeron and
Sottile \cite{ABS}, gives a natural algebraic framework of
enumerative combinatorics. The extensive survey of Hopf algebra
theory in combinatorics may be found in \cite{GR}.

A combinatorial Hopf algebra is a graded, connected Hopf algebra
$\mathcal{H}$ equipped with a multiplicative linear functional
$\zeta:\mathcal{H}\rightarrow\mathbf{k}$ to the ground field. We
describe the combinatorial Hopf algebra $QSym$ of quasisymmetric
functions.


A quasisymmetric function $F=F(\mathbf{x})$ is a formal power
series of bounded degree in the countable ordered set of variables
$\mathbf{x}=(x_1,x_2,\ldots)$ such that coefficients by monomials
with the same list of ordered exponents are equal. This condition
produces the natural linear basis for the algebra $QSym$
consisting of monomial quasisymmetric functions

$$M_\alpha=\displaystyle\sum_{i_1<i_2<\cdots<i_k}x_{i_1}^{a_1}x_{i_2}^{a_2}\cdots
x_{i_k}^{a_k},$$ indexed by finite ordered sets of integers
$\alpha=(a_1,a_2,\ldots,a_k)$ called compositions of
$|\alpha|=a_1+a_2+\cdots a_k$ of the length $k(\alpha)=k$. The
coproduct, defined on monomial basis by

$$\Delta(M_\alpha)=\displaystyle\sum_{\beta\gamma=\alpha}M_\beta\otimes
M_\gamma,$$ where $\beta\gamma$ is the concatenation of
compositions, turns $QSym$ into a graded, connected Hopf algebra.


The principal specialization
$\mathbf{ps}:QSym\rightarrow\mathbf{k}[m]$ assigns to a
quasisymmetric function $F$ a polynomial in $m$ by evaluation map
$$\mathbf{ps}(F)(m)=F|_{x_1=\cdots=x_m=1,x_{m+1}=\cdots=0}.$$
The canonical character on quasisymmetric functions
$\zeta_{\mathcal{Q}}:QSym\rightarrow\mathbf{k}$ is defined by

$$\zeta_{\mathcal{Q}}(F)=\mathbf{ps}(F)(1).$$ It is easy to see that
on monomial basis we have
$$\mathbf{ps}(M_\alpha)(m)={m \choose k(\alpha)}$$ and specially

\begin{equation}\label{-1}
\mathbf{ps}(M_\alpha)(-1)=(-1)^{k(\alpha)}.
\end{equation}


The following theorem is fundamental in applications and explains
the ubiquity of quasisymmetric functions as enumerator functions
in combinatorics.

\begin{theorem}[\cite{ABS}, Theorem 4.1]\label{fundamental}
For a combinatorial Hopf algebra $(\mathcal{H},\zeta)$ there is a
unique morphism of graded Hopf algebras
$$\Psi:\mathcal{H}\rightarrow QSym$$
such that $\Psi\circ\zeta_{\mathcal{Q}}=\zeta$. For a homogeneous
element $h$ of degree $n$ the coefficients $\zeta_\alpha(h),
\alpha=(a_1,\ldots,a_k)$ of $\Psi(h)$ in monomial basis of $QSym$
are given by

$$\zeta_\alpha(h)=\zeta^{\otimes k}\circ(p_{a_1}\otimes
p_{a_2}\otimes\cdots\otimes p_{a_k})\circ\Delta^{(k-1)}(h),$$
where $p_i$ is the projection on the $i$-th homogeneous component
and $\Delta^{(k-1)}$ is the $(k-1)$-fold coproduct map of
$\mathcal{H}$.
\end{theorem}

\section{Weighted quasisymmetric enumerator $F_q(Q)$}

A generalized permutohedron $Q$ comes with a natural map between
face lattices $\pi_Q:L(Pe^{n-1})\rightarrow L(Q)$ given by
Definition \ref{lattice} with $\pi_Q(\mathcal{F})=G$ if and only
if $\sigma_\mathcal{F}^{\circ}\subset\sigma_G^{\circ}$. This map
produces a natural statistic of faces of the standard
permutohedron $Pe^{n-1}$.

\begin{definition}\label{Q-rank}
For a generalized permutohedron $Q$ the $Q$-{\it rank} is a map on
the face lattice of the standard permutohedron
$\mathrm{rk}_Q:L(Pe^{n-1})\rightarrow\{0,1,\ldots,n-1\}$ given by
$$\mathrm{rk}_Q(\mathcal{F})=\mathrm{dim}(\pi_Q(\mathcal{F})).$$
\end{definition}

Let $\omega=(\omega_1,\omega_2,\ldots,\omega_n)\in\mathbb{Z}_+^n$
be an integer lattice vector with positive entries. It defines the
\emph{weight function} $\omega^{\star}:Q\rightarrow\mathbb{R}$ on
$Q$ by $\omega^{\star}(x)=\langle\omega,x\rangle$, where $\langle
, \rangle$ is the standard scalar product in $\mathbb{R}^{n}$.
Note that $Q$ lies in a hyperplane whose normal vector is
$(1,\ldots,1)$. The weight function $\omega^{\star}$ is maximized
along a unique face $G_\omega$ of $Q$ which is determined by the
condition that the vector $\omega$ lies in the relative interior
of its normal cone $\omega\in\sigma_{G_\omega}^{\circ}$. A weight
function $\omega^{\star}$ is called $Q$-{\it generic} if
$G_\omega$ is a vertex of $Q$.

If $Q=Pe^{n-1}$ is the standard permutohedron each
$\omega\in\mathbb{Z}_+^{n}$ determines a unique flag
$\mathcal{F}_\omega$ by the condition
$\mathcal{F}=\mathcal{F}_\omega$ if and only if the integer vector
$\omega$ lies in the relative interior of the corresponding braid
cone $\omega\in\sigma_\mathcal{F}^{\circ}$.

\begin{definition}
For a generalized permutohedron $Q$ let $F_q(Q)$ be a weighted
enumerator of positive integer vectors
$$F_q(Q)=\sum_{\omega\in\mathbb{Z}_+^{n}}q^{\mathrm{rk}_Q(\mathcal{F}_\omega)}\mathbf{x}_\omega.$$
\end{definition}

We expand the enumerator function $F_q(Q)$ in the monomial bases
of quasisymmetric functions.

\begin{definition}
For a flag $\mathcal{F}$ of subsets of $[n]$ let $M_\mathcal{F}$
be the enumerator of positive integer vectors in relative interior
of the corresponding braid cone
$$M_\mathcal{F}=\displaystyle\sum_{\omega\in\mathbb{Z}^{n}_+\cap\sigma_\mathcal{F}^{\circ}}\mathbf{x}_\omega.$$
\end{definition}
An enumerator $M_\mathcal{F}$ is exactly the monomial
quasisymmetric function depending only on the type of
$\mathcal{F}$, given by $(\ref{type})$

$$M_\mathcal{F}=M_{\mathrm{type}(\mathcal{F})}.$$
We obtain the expansion of $F_q(Q)$ according to the face lattice
of the standard permutohedron
\begin{equation}\label{expansion}
F_q(Q)=\displaystyle\sum_{\mathcal{F}\in
L(Pe^{n-1})}q^{\mathrm{rk}_Q(\mathcal{F})}M_\mathcal{F}.
\end{equation}
In the monomial basis $F_q(Q)$ has the expansion of the form
$$F_q(Q)=\sum_{\alpha\models n}p_\alpha(q)M_\alpha$$ where
$p_\alpha(q)$ are polynomials in $q$ indexed by compositions of
$n$. For a composition $\alpha$ the polynomial $p_\alpha(q)$ is
given by

$$p_\alpha(q)=\sum_{\mathcal{F}:\mathrm{type}(\mathcal{F})=\alpha}q^{\mathrm{rk}_Q(\mathcal{F})}.$$

By definition \ref{lattice} of the map $\pi_Q$ we have

$$\sum_{\omega\in\sigma_G^{\circ}}\mathbf{x}_\omega=\sum_{\mathcal{F}:\pi_Q(\mathcal{F})=G}M_\mathcal{F},$$
which gives the expansion of $F_q(Q)$ in terms of the face lattice
of $Q$

\begin{equation}\label{poQ}
F_q(Q)=\sum_{G\in
L(Q)}q^{\mathrm{dim}(G)}\sum_{\mathcal{F}:\pi_Q(\mathcal{F})=G}M_\mathcal{F}.
\end{equation}
From this we easily derive that the enumerator $F_q$ contains the
information about $f$-vectors of generalized permutohedra. The
$f$-vector of a convex $(n-1)$-dimensional polytope $P$ is the
integer vector $f=(f_0,f_1,\ldots,f_{n-1}),$ where $f_i$ is the
number of $i$-dimensional faces of $P$. It is codified by the
$f$-polynomial

$$f(P,q)=f_0+f_1q+f_2q^{2}+\cdots+f_{n-1}q^{n-1}.$$

\begin{theorem}\label{general}
The $f$-polynomial $f(Q,q)$ of an $(n-1)$-dimensional generalized
permutohedron $Q$ is determined by the principal specialization
$$f(Q,q)=(-1)^{n}\mathbf{ps}(F_{-q}(Q))(-1).$$
\end{theorem}
\begin{proof}
By formula $(\ref{poQ})$ and the fact
$\mathbf{ps}(M_\mathcal{F})(-1)=(-1)^{|\mathcal{F}|+1}$ implied by
$(\ref{-1})$ the principal specialization of $F_q(Q)$ gives

$$(-1)^{n}\mathbf{ps}(F_{-q}(Q))(-1)=\displaystyle\sum_{G\in L(Q)}q^{\dim(G)}
\displaystyle\sum_{\mathcal{F}:\pi_Q(\mathcal{F})=G}(-1)^{|\mathcal{F}|+1+n+\dim(G)}.$$
The inner sum is equal to 1 by Proposition \ref{euler}, which
completes the proof.
\end{proof}

\subsection{Action of the antipode on $F_q(Q)$}

In this subsection we determine how the antipode $S$ of the Hopf
algebra of quasisymmetric functions $QSym$ acts on the weighted
quasisymmetric enumerator function $F_q(Q)$.

Define the opposite flag $\mathcal{F}^{op}$ to a flag
$\mathcal{F}:\emptyset=:F_0\subset F_1\subset\ldots\subset
F_{k}\subset F_{k+1}:=[n]$ by

$$\mathcal{F}^{op}:\emptyset\subset [n]\setminus
F_k\subset\cdots\subset[n]\setminus F_1\subset[n].$$ The
corresponding braid cones are related by
$\sigma_{\mathcal{F}^{op}}=-\sigma_\mathcal{F}.$ For a composition
$\alpha=(a_1,\ldots,a_k)$ the reverse composition is
$\mathrm{rev}(\alpha)=(a_k,\ldots,a_1)$. We have

\begin{equation}\label{rev}
\mathrm{type}(\mathcal{F}^{op})=\mathrm{rev}(\mathrm{type}(\mathcal{F})).
\end{equation}
Recall that flags are ordered by refinements. The following lemma
gives a particulary nice geometric meaning of the formula for the
antipode $S$, see \cite[Theorem 5.11]{GR} and reference within.

\begin{lemma}\label{antipode}
The antipode $S$ on the monomial quasisymmetric function
$M_\mathcal{F}$ associated to a flag $\mathcal{F}$ acts by
$$S(M_\mathcal{F})=(-1)^{|\mathcal{F}|+1}\sum_{\mathcal{G}\preceq\mathcal{F}^{op}}M_\mathcal{G}.$$
\end{lemma}
\noindent This allows us to interpret $S(M_\mathcal{F})$ as the
enumerator function of integer lattice points lying in the
opposite braid cone $\sigma_{\mathcal{F}^{op}}$

$$(-1)^{|\mathcal{F}|+1}S(M_\mathcal{F})=\sum_{\omega\in\mathbb{Z}^{n}_+\cap\sigma_{\mathcal{F}^{op}}}\mathbf{x}_\omega.$$

Let $\pi_Q$ be the map associated to a generalized permutohedron
$Q$ by definition \ref{lattice}. We say that the face
$\pi_Q(\mathcal{F}^{op})$ is opposite to a flag $\mathcal{F}$. The
following theorem describes the action of the antipode on the
weighted quasisymmetric enumerator function $F_q(Q)$.

\begin{theorem}\label{generalantipode}
Given a generalized permutohedron $Q$ of dimension $n-1$ the
antipode $S$ acts on the weighted quasisymmetric enumerator
function $F_q(Q)$ by

$$S(F_q(Q))=(-1)^{n}\sum_{\mathcal{G}}f(\pi_Q(\mathcal{G}^{op}),-q)M_\mathcal{G},$$
where the sum is over all flags $\mathcal{G}$ of the set $[n]$ and
$f(\pi_Q(\mathcal{G}^{op}),q)$ is the $f$-polynomial of the face
$\pi_Q(\mathcal{G}^{op})$ opposite to a flag $\mathcal{G}$.
\end{theorem}

\begin{proof}

By the expansion $(\ref{expansion})$ and lemma \ref{antipode} we
have
$$S(F_q(Q))=\sum_\mathcal{F}q^{\mathrm{rk}_Q(\mathcal{F})}S(M_\mathcal{F})=
\sum_\mathcal{F}q^{\mathrm{rk}_Q(\mathcal{F})}(-1)^{|\mathcal{F}|+1}\sum_{\mathcal{G}\preceq\mathcal{F}^{op}}M_\mathcal{G},$$

which gives

$$S(F_q(Q))=\sum_\mathcal{G}M_\mathcal{G}\sum_{\mathcal{F}:\mathcal{G}^{op}\preceq\mathcal{F}}(-1)^{|\mathcal{F}|+1}q^{\mathrm{rk}_Q(\mathcal{F})}.$$
It remains to show that
$$\sum_{\mathcal{F}:\mathcal{G}^{op}\preceq\mathcal{F}}(-1)^{|\mathcal{F}|+1}q^{\mathrm{rk}_Q(\mathcal{F})}=(-1)^{n}f(\pi_Q(\mathcal{G}^{op}),-q),$$
i.e.

$$f(\pi_Q(\mathcal{G}^{op}),q)=\sum_{\mathcal{F}:\mathcal{G}^{op}\preceq\mathcal{F}}
(-1)^{|\mathcal{F}|+1+n+\mathrm{rk}_Q(\mathcal{F})}q^{\mathrm{rk}_Q(\mathcal{F})},$$
which is a consequence of Definition \ref{Q-rank} and Proposition
\ref{euler}.

\end{proof}

Theorem \ref{generalantipode} generalizes Theorem \ref{general}.
To see this note that $\pi_Q(\emptyset\subset[n])=Q$ and by
Theorem \ref{generalantipode} the coefficient in $S(F_q(Q))$ by
monomial quasisymmetric function $M_n=M_{\emptyset\subset[n]}$ is
equal to $(-1)^{n}f(Q,-q)$. This coefficient can be extracted from
$S(F_q(Q))$ by composing with the canonical character
$\zeta_{\mathcal{Q}}$ on quasisymmetric functions. Theorem
\ref{general} then follows from the fact

\begin{equation}\label{zetaS}
\zeta_\mathcal{Q}\circ S(F)=\mathbf{ps}(F)(-1),
\end{equation}
which can be easily seen to hold for monomial bases and
consequently for each quasisymmetric function $F$.

\begin{corollary}
For a generalized permutohedron $Q$ the following identity holds
$$\mathbf{ps}(S(F_q(Q)))(-1)=q^{\mathrm{dim}(Q)}.$$
\end{corollary}
\begin{proof}
It follows from the equation $(\ref{zetaS})$ and the fact that the
antipode $S$ of $QSym$ is of order two $S^{2}=\mathrm{Id}$.
\end{proof}

\subsection{Specializations of the enumerator $F_q(Q)$}

We review some specializations of the quasisymmetric enumerator
function $F_q(Q)$. By taking $q=0$ it is specified to the Billera,
Jia and Reiner quasisymmetric function $F(Q)$ which is the
enumerator of $Q$-generic weight vectors

$$F(Q)=\sum_{\omega\in\mathbb{Z}_+^{n}:\mathrm{rk}(\mathcal{F}_\omega)=0}\mathbf{x}_\omega.$$
Theorem \ref{general} may be regarded as the generalization of the
formula for number of vertices $f_0$ of a generalized
permutohedron $Q$, see \cite[Theorem 9.2]{BJR}

$$f_0=(-1)^{n}\mathbf{ps}(F(Q))(-1).$$
Note that $F_q(Q)$ degenerates at $q=1$ since
$F_1(Q)=\sum_{\omega\in\mathbb{Z}_+^{n}}\mathbf{x}_\omega=(M_1)^{n}.$
Theorem \ref{general} then gives

$$f(Q,-1)=f_0-f_1+\cdots+(-1)^{n-1}f_{n-1}=(-1)^{n}\mathbf{ps}(M_1^{n})(-1)=1,$$
which is Euler characteristic of $Q$.

For particular classes of generalized permutohedra the enumerator
$F_q$ specializes to known quasisymmetric invariants as it is
announced in introduction. We briefly review the case of graphical
zonotopes and more thoroughly the cases of nestohedra and matroid
base polytopes in subsequent sections.

\paragraph{Grafical zonotopes}

The vertices of a graphical zonotope $Z_\Gamma$ are in one-to-one
correspondence with regions of the corresponding graphical
arrangement. The integer points $\omega\in\mathbb{Z}_+^{n}$ can be
interpreted as graph colorings by labelling vertices of a graph
$\Gamma$. The normal cone at a vertex of $Z_\Gamma$ is a region of
the corresponding graphical arrangement. So $\omega$ determines a
proper coloring of the graph $\Gamma$ if and only if
$\mathrm{rk}_{Z_\Gamma}(\mathcal{F}_\omega)=0$. Thus the
enumerators of proper colorings and $Z_\Gamma$-generic integer
vectors coincide, i.e. $F_0(Z_\Gamma)=X_\Gamma$, where $X_\Gamma$
is the Stanley chromatic symmetric function (introduced in
\cite{S1}).

The main argument in \cite{GV2} of the proof of Theorem
\ref{general} for graphical zonotopes was based on the Humpert and
Martin cancelation-free formula for the antipode of the chromatic
Hopf algebra of graphs \cite{HM}.


\section{Nestohedra}

We refer the reader to \cite{PRW} for definitions and main
properties of nestohedra. The nestohedra are a class of simple
polytopes described by the notion of building sets. A collection
of subsets $B$ of a finite ground set $V$ is a building set if

$\diamond$ $\{i\}\in B$ for all $i\in V$ and

$\diamond$ if $I, J\in B$ and $I\cap J\neq\emptyset$ then $I\cup
J\in B$.

A building set $B$ is connected if $V\in B$. Let
$\Delta^{n-1}=\mathrm{Conv}\{e_1,\ldots,e_n\}$ be the standard
simplex in $\mathbb{R}^{n}$. The nestohedron associated to a
building set $B$ on the ground set $[n]$ is the Minkowsky sum of
simpleces $P_B=\sum_{I\in B}{\rm Conv}\{e_i\mid i\in I\}$.
Enumerate faces of $\Delta^{n-1}$ by subsets of $[n]$ in a way
that the face poset of $\Delta^{n-1}$ is isomorphic to the reverse
Boolean lattice on $[n]$. For a connected building set $B$ the
nestohedron $P_B$ is realized by successive truncations over faces
of $\Delta^{n-1}$ encoded by a building set $B$ in any
nondecreasing sequence of dimensions of faces. Thus facets of
$P_B$ are labelled by elements $I\in B\setminus\{[n]\}$. Recall
that a truncation of a convex polytope $P$ over a face $F\subset
P$ is the polytope $P\setminus F$ obtained by cutting $P$ with a
hyperplane $H_F$ which divides vertices in $F$ and vertices not in
$F$ in separated halfspaces. For a disconnected building set $B$
the associated nestohedron $P_B$ is the product of nestohedra
corresponding to components of $B$.

The following condition describes the face poset of a nestohedron
$P_B$ corresponding to a connected building set $B$, see
\cite{FS}, Theorem 3.14 and \cite{P}, Theorem 7.4. The
intersection $F_{I_1}\cap\ldots\cap F_{I_k}, k\geq2$ of facets
corresponding to a subcollection $N=\{I_1,\ldots,I_k\}\subset
B\setminus\{[n]\}$ is a nonempty face of $P_B$ if and only if
\begin{itemize}
\item[({\rm N1})] $I_i\subset I_j$ or $I_j\subset I_i$ or $I_i\cap
I_j=\emptyset$ for any $1\leq i<j\leq k$, \item[({\rm N2})]
$I_{j_1}\cup\cdots\cup I_{j_p}\notin B$ for any pairwise disjoint
sets $I_{j_1},\ldots,I_{j_p}$.
\end{itemize}
Subcollections that satisfy conditions (N1) and (N2) form a
simplicial complex, called the {\it nested set complex}, whose
face poset is opposite to the face poset of $P_B$.

\subsection{Hopf algebra $\mathcal{B}$}

Two building sets $B_1$ and $B_2$ are isomorphic if there is a
bijection of their sets of vertices $f:V_1\rightarrow V_2$ such
that $I\in B_1$ if and only if $f(I)\in B_2$. The addition of
building sets $B_1$ and $B_2$ on disjoint ground sets $V_1$ and
$V_2$ is the building set $B_1\sqcup B_2=\{I\subset V_1\sqcup
V_2\mid I\in B_1 \ \mbox{or} \ I\in B_2\}$. For a building set $B$
on $V$ and a subset $S\subset V$ the restriction on $S$ and the
contraction of $S$ from $B$ are defined by $B\mid_S=\{I\subset
S\mid I\in B\}$ and $B/S=\{I\subset V\setminus S\mid I\in B \
\mbox{or} \ I\cup S'\in B \ \mbox{for some} \ S'\subset S\}$. The
building sets obtained from $B$ by restrictions and contractions
are its minors.

The following combinatorial Hopf algebra of building set is
considered in \cite{GV1}. The set of all isomorphism classes of
finite building sets linearly generates the vector space
$\mathcal{B}$ over a field $\mathbf{k}$. The space $\mathcal{B}$
is a graded, commutative and non-cocommutative Hopf algebra with
the multiplication and the comultiplication
\[[B_1]\cdot [B_2]=[B_1\sqcup B_2] \ \ \mbox{and} \ \
\Delta([B])=\sum_{S\subset V}[B\mid_S]\otimes [B/S].\] The grading
$\mathrm{gr}([B])$ is given by the cardinality of the ground set
of $B$. A building set $B$ is connected if $[B]$ is irreducible,
i.e. it is not represented by an addition of two building sets.
Denote by $c(B)$ the number of connected components of $B$. Let
$\zeta:\mathcal{B}\rightarrow\mathbf{k}$ be a multiplicative
linear functional defined by $\zeta([B])=1$ if $B$ is a discrete
(consisting of singletons only) and $\zeta([B])=0$ otherwise. A
unique morphism $\Psi:(\mathcal{B},\zeta)\rightarrow
(QSym,\zeta_Q)$ of combinatorial Hopf algebras is given in the
monomial basis of quasismmetric functions by
\[\Psi([B])=\sum_{\alpha\models\mathrm{gr}(B)}\zeta_\alpha(B)M_\alpha.\]
The coefficients $\zeta_\alpha(B)$ have an enumerative meaning.
Let $\mathcal{L}:\emptyset\subset I_1\subset\cdots\subset I_k=V$
be a chain of subsets of the ground set $[n]$. Denote by
$|\mathcal{L}|=k$ its length and by $\mathrm{type}(\mathcal{L})$
its type which is a composition $\alpha=(i_1,\ldots,i_k)$ such
that for any $1\leq j\leq k$ the set $I_j\setminus I_{j-1}$ has
$i_j$ elements. We say that $\mathcal{L}$ is a splitting chain if
all minors $B\mid_{I_j}/I_{j-1}$ are discrete. Then
$\zeta_\alpha(B)$ is exactly the number of all splitting chains of
$B$ of a given type $\alpha$. For a building set $B$ on $[n]$ the
following identity holds (\cite[Theorem 4.5]{GV1})

\begin{equation}\label{eqn:eqn1}
F(P_B)=\Psi([B]).
\end{equation}

\subsection{$q$-analog}

We extend the basic field $\mathbf{k}$ into the field of rational
functions $\mathbf{k}(q)$ and define the character
$\zeta_q:\mathcal{B}\rightarrow\mathbf{k}(q)$ with
$\zeta_q([B])=q^{\mathrm{rk}(B)}$, where
$\mathrm{rk}(B)=\mathrm{gr}(B)-c(B)$. Let $\Psi_q:(\mathcal{B},
\zeta_q)\rightarrow (QSym, \zeta_Q)$ be a unique morphism of
combinatorial Hopf algebras over $\mathbf{k}(q)$.

Recall that a reflexive and transitive relation $\preccurlyeq$ on
$[n]$ is called a {\it preorder}. If in addition $p\preccurlyeq q$
or $q\preccurlyeq p$ for any $p,q\in [n]$ it is called a {\it weak
order} on $[n]$. Weak orders correspond to set compositions of
$[n]$. Each set composition determines a unique flag of subsets
and vise-versa. We say that a flag of subsets
$\mathcal{L}:\emptyset=I_0\subset I_1\subset\ldots\subset I_k=[n]$
extends a preorder $\preccurlyeq$ on $[n]$ if

\[i\prec j \ \ \mbox{implies} \ \ i\in I_p\setminus I_{p-1} \ \
\mbox{and} \ \ j\in I_q\setminus I_{q-1} \ \ \mbox{for some} \ \
1\leq p< q\leq k,\] where $i\prec j$ means that $i\preccurlyeq j$
and it is not $j\preccurlyeq i$.

Let $G=F_{I_1}\cap\ldots\cap F_{I_m}$ be a face of $P_B$ and
$N_G=\{I_1,\ldots,I_m\}\subset B\setminus\{[n]\}$ the
corresponding nested set. Note that $\dim G=n-1-m$. For each $I\in
N_G\cup\{[n]\}$ let $I^{\mathrm{root}}$ be the set of roots of $I$
given by

\[I^{\mathrm{root}}=I\setminus\cup\{J\in N_G\mid
J\varsubsetneq I\}.\]  Define a preorder $\preccurlyeq_G$ on $[n]$
corresponding to the face $G\subset P_B$ by

\[i\preccurlyeq_G j \ \ \mbox{if and only if} \ \ i\in I, j\in
I^{\mathrm{root}} \ \ \mbox{for some} \ \ I\in N_G\cup\{[n]\}.\]
We describe the map among face lattices
$\pi_{P_B}:L(Pe^{n-1})\rightarrow L(P_B)$ given by Definition
\ref{lattice}.

\begin{lemma}
For a face $G$ of a nestohedron $P_B$ and a flag of subsets
$\mathcal{F}$ on $[n]$ we have $\pi_{P_B}(\mathcal{F})=G$ if and
only if a flag $\mathcal{F}$ extends the preorder
$\preccurlyeq_G$.
\end{lemma}
\begin{proof}
The proof follows from the coordinate descriptions of
corresponding normal cones. The normal cone $\sigma_G$ at the face
$G$ is given by system of inequalities

\[\sigma_G: x_i\leq x_j \ \ \mbox{if and only if} \ \
i\preccurlyeq_G j \ \ \mbox{for} \ \ i,j\in[n].\] The braid cone $
\sigma_\mathcal{F}$ corresponding to a flag $\mathcal{F}$ is given
by relations $(\ref{braidcone})$. We see that
$\sigma_\mathcal{F}^{\circ}\subset\sigma_G^{\circ}$ if and only if
$\mathcal{F}$ extends $\preccurlyeq_G$.
\end{proof}

For a connected building set $B$ on $[n]$ and a flag of subsets
$\mathcal{L}:\emptyset=I_0\subset I_1\subset\ldots\subset I_k=[n]$
denote by $B/\mathcal{L}$ the {\it quotient building set}

$$B/\mathcal{L}=\bigoplus_{j=1}^{k}\left(B|_{I_{j}}/I_{j-1}\right).$$
Define the rank of $\mathcal{L}$ according to $B$ as

$$\mathrm{rk}_B(\mathcal{L})=\sum_{j=1}^{k}\mathrm{rk}(B\mid_{I_j}/I_{j-1}).$$
Note that $\mathrm{rk}_B(\mathcal{L})=n-c(B/\mathcal{L})$, where
$c(B/\mathcal{L})$ is the number of components of the quotient
building set $B/\mathcal{L}$.

\begin{proposition}\label{ranks}
For a connected building set $B$ on $[n]$ and a flag $\mathcal{L}$
of subsets of $[n]$ the ranks $\mathrm{rk}_{P_B}(\mathcal{L})$ and
$\mathrm{rk}_B(\mathcal{L})$ coincide

$$\mathrm{rk}_{P_B}(\mathcal{L})=\mathrm{rk}_B(\mathcal{L}).$$
\end{proposition}
\begin{proof}
Let $G$ be a proper face of $P_B$ of dimension $\mathrm{dim} G=m$
and $\mathcal{L}:\emptyset=I_0\subset I_1\subset\ldots\subset
I_k=[n]$ be a flag of subsets of $[n]$ such that
$\pi_{P_B}(\mathcal{L})=G$. To the face $G$ corresponds the nested
set $N_G=\{J_1,\ldots,J_{n-m-1}\}\subset B\setminus\{[n]\}$. A
simple argumentation shows that

$$c(B\mid_{I_p}/I_{p-1})=|\{J\in N_G\cup\{[n]\}\mid
J^{\mathrm{root}}\subset I_p\setminus I_{p-1}\}| \ \ \mbox{for
all} \ \ 1\leq p\leq k.$$ Since for any $J\in N_G\cup\{[n]\}$
there is a unique $p$ such that $J^{\mathrm{root}}\subset
I_p\setminus I_{p-1}$ we have

$$\sum_{j=1}^{k}c(B\mid_{I_j}/I_{j-1})=n-m.$$

\end{proof}

The following theorem gives an algebraic characterization of the
quasisymmetric enumerator function $F_q(P_B)$.

\begin{theorem}\label{coincides}
For a building set $B$ the quasisymmetric enumerator function
$F_q(P_B)$ associated to a nestohedron $P_B$ coincides with the
value of a universal morphism  $$F_q(P_B)=\Psi_q([B]).$$
\end{theorem}

\begin{proof}

The morphism $\Psi_q$ is given by
$$\Psi_q([B])=\sum_{\alpha\models\mathrm{gr}(B)}(\zeta_q)_\alpha(B)M_\alpha.$$
By Theorem \ref{fundamental} the coefficient corresponding to a
composition $\alpha=(i_1,\ldots,i_k)\models n$ is determined by

\begin{equation}\label{eqn:coeff}
(\zeta_q)_\alpha(B)=\sum_{\mathcal{L}:\mathrm{type}(\mathcal{L})=\alpha}\prod_{j=1}^{k}q^{\mathrm{rk}(B\mid_{I_j}/I_{j-1})}
=\sum_{\mathcal{L}:\mathrm{type}(\mathcal{L})=\alpha}q^{\mathrm{rk}_B(\mathcal{L})},
\end{equation}
where the sum is over all chains $\mathcal{L}:\emptyset\subset
I_1\subset\ldots\subset I_k=V$ of the type $\alpha$ and
$\mathrm{rk}_B(\mathcal{L})$ is the rank of the flag $\mathcal{L}$
according to $B$. Thus

\begin{equation}\label{eqn:eqn3}
F_q(P_B)=\sum_{\mathcal{L}}q^{\mathrm{rk}_B(\mathcal{L})}M_{\mathrm{type}(\mathcal{L})},
\end{equation}
where the sum is over all chains of the ground set $[n]$. Theorem
follows according to the expansion $(\ref{expansion})$ and
Proposition \ref{ranks}.
\end{proof}

\begin{example}
The permutohedron $Pe^{n-1}=P_{B_n}$ is realized as the
nestohedron corresponding to the family $B_n$ of all subsets of
$[n]$. Since $\mathrm{rk}_{B_n}(\mathcal{L})=n-|\mathcal{L}|$ for
any chain $\mathcal{L}$ of subsets of $[n]$ we have by
$(\ref{eqn:eqn3})$ that
\[F_q(Pe^{n-1})=\sum_{\mathcal{L}}q^{n-|\mathcal{L}|}M_{\mathrm{type}(\mathcal{L})}=\sum_{\alpha\models n}{n \choose \alpha}q^{n-k(\alpha)}M_\alpha.\]
Consequently by Theorem \ref{general} we derive the well known
fact
\[f(Pe^{n-1},q)=\sum_{\alpha\models n}{n \choose
\alpha}q^{n-k(\alpha)}.\]
\end{example}

\begin{example}
For the building set $B=\{\{1\},\ldots,\{n\},[n]\}$ on $[n]$ the
corresponding nestohedron is the $(n-1)$-simplex
$P_B=\Delta^{n-1}$. Let $\mathcal{L}$ be a chain of the type
$\mathrm{type}(\mathcal{L})=\alpha\models n$. Obviously
$\mathrm{rk}_B(\mathcal{L})=l(\alpha)-1$, where $l(\alpha)$
denotes the last component of the composition $\alpha\models n$.
Therefore by $(\ref{eqn:coeff})$ we have $(\zeta_q)_\alpha(B)={n
\choose \alpha}q^{l(\alpha)-1},$ and consequently

\[F_q(\Delta^{n-1})=\sum_{\alpha\models n}{n \choose
\alpha}q^{l(\alpha)-1}M_\alpha.\] By rearranging summands
according to last components of compositions we obtain

\[F_q(B)=\sum_{k=1}^{n}{n \choose k}q^{k-1}\sum_{\alpha\models
n-k}{n-k \choose \alpha}M_{(\alpha,k)}.\] Taking into account that
$M_{(1)}^{n}=\sum_{\alpha\models n}{n \choose \alpha}M_\alpha,$
for each $n$ we have

\[F_q(\Delta^{n-1})=\sum_{k=1}^{n}{n \choose
k}q^{k-1}(M_{(1)}^{n-k})_k.\] Theorem \ref{general} gives the
expected

\[f(\Delta^{n-1},q)=\sum_{k=1}^{n}{n \choose
k}q^{k-1}=\frac{(1+q)^{n}-1}{q}.\]

\end{example}



\subsection{Recursive behavior of $F_q(P_B)$}

For a composition $\alpha=(a_1,\ldots,a_k)$ and a positive integer
$r$ let $(\alpha,r)=(a_1,\ldots,a_k,r)$. Define a shifting
operator $F\mapsto (F)_r$ on $QSym$ as the linear extension of the
map given on the monomial basis by $M_\alpha\mapsto
M_{(\alpha,r)}$. Specially
$(M_\emptyset)_r=M_{(r)}=x_1^{r}+x_2^{r}+\cdots$.




The next theorem shows that the weighted quasisymmetric enumerator
function for nestohedra $F_q(P_B)$ is determined by recurrence
relations.


\begin{theorem}\label{reccurence}
The quasisymmetric function $F_q(B)=F_q(P_B)$ is determined by the
following recurrence relations
\begin{item}
\item[(1)] $F_q(\bullet)=M_{(1)}$ for the singleton
$\bullet=\{\{1\}\}.$

\item[(2)] If $B=B_1\sqcup B_2$ then $F_q(B)=F_q(B_1)F_q(B_2).$

\item[(3)] If $B$ is connected then $F_q(B)=\sum_{I\varsubsetneq
[n]}q^{n-|I|-1}(F_q(B\mid_I))_{n-|I|}.$

\end{item}
\end{theorem}

\begin{proof}
The assertions (1) and (2) are direct consequences of definition
of the enumerator $F_q(B)$. It remains to prove the assertion (3).
Note that for connected $B$ the contraction $B/I$ remains
connected for each $I\subset [n]$ and $\mathrm{rk}(B/I)=n-|I|-1$.
If we rearrange the sum in the expansion $(\ref{eqn:eqn3})$
according to predecessors of the maximal elements in chains we
obtain

\[F_q(B)=\sum_{I\subsetneq
[n]}q^{n-|I|-1}\sum_{\mathcal{L}_I}q^{\mathrm{rk}_{B\mid_I}(\mathcal{L}_I)}M_{(\mathrm{type}(\mathcal{L}_{I}),n-|I|)},\]
where the last sum is over all chains $\mathcal{L}_I$ of $I$. This
leads, by repeated application of equation $(\ref{eqn:eqn3})$ to
the needed identity.
\end{proof}

The recursive behavior of the enumerator $F_q(P_B)$ together with
Theorem \ref{general} reproves the recursive behavior of
$f$-polynomials of nestorhedra.

\begin{theorem}[\cite{P}, Theorem 7.11]\label{thm:thm2}
The $f$-polynomial $f(B,q)$ of a nestohedron $P_B$ is determined
by the following recurrence relations
\begin{item}
\item[(1)] $f(\bullet,q)=1$ for the singleton $\bullet=\{\{1\}\}.$

\item[(2)] If $B=B_1\sqcup B_2$ then $f(B,q)=f(B_1,q)f(B_2,q).$

\item[(3)] If $B$ is connected then $f(B,q)=\sum_{I\varsubsetneq
[n]}q^{n-|I|-1}f(B\mid_I,q).$

\end{item}
\end{theorem}

\section{Graph-associahedra}

A special class of building sets is produced by simple graphs. The
graphical building set $B(\Gamma)$ on a graph $\Gamma$ is the
collection of all subsets of vertices such that induced subgraphs
are connected. The polytope $P_\Gamma=P_{B(\Gamma)}$ is called a
{\it graph-associahedron}.

In \cite{GV1} is considered the following Hopf algebra of graphs.
Let $\mathcal{G}$ be a vector space over the field $\mathbf{k}$
spanned by isomorphism classes of simple graphs. It is endowed
with a Hopf algebra structure by operations

\[[\Gamma_1]\cdot[\Gamma_2]=[\Gamma_1\sqcup\Gamma_2] \ \mbox{and} \
\Delta([\Gamma])=\sum_{I\subset
V}[\Gamma\mid_I]\otimes[\Gamma/I],\] where $\Gamma\mid_I$ is the
induced subgraph on $I$ and $\Gamma/I$ is the induced subgraph on
$V\setminus I$ with additional edges $uv$ for all pairs of
vertices $u,v\notin I$ connected by edge paths through $I$. The
correspondence $\Gamma\mapsto B(\Gamma)$ defines a Hopf
monomorphism from $\mathcal{G}$ to $\mathcal{B}$.

The induced character on $\mathcal{G}$ over the field of rational
functions $\mathbf{k}(q)$ is given by
$\zeta_q(\Gamma)=q^{\mathrm{rk}(\Gamma)}$, where
$\mathrm{rk}(\Gamma)=\mathrm{gr}(\Gamma)-c(\Gamma)$ with
$\mathrm{gr}(\Gamma)$ and $c(\Gamma)$ being the numbers of
vertices and connected components of the graph $\Gamma$. This
results in a $q$-analog $F_q(\Gamma)$ of the quasisymmetric
function invariant $F(\Gamma)=F(P_\Gamma)$ of the graph $\Gamma$
introduced and studied in \cite{GV1}. By formula
$(\ref{eqn:eqn3})$ the qusisymmetric function $F_q(\Gamma)$ is
described by

\begin{equation}\label{qgraph}
F_q(\Gamma)=\sum_{\mathcal{L}}q^{\mathrm{rk}_\Gamma(\mathcal{L})}M_{\mathrm{type}(\mathcal{L})},
\end{equation}
where
$\mathrm{rk}_\Gamma(\mathcal{L})=\sum_{j=1}^{k}\mathrm{rk}(\Gamma\mid_{I_j}/I_{j-1})$
for a chain $\mathcal{L}:\emptyset\subset I_1\subset\cdots\subset
I_k=V$.

Theorem \ref{reccurence} applied on the case of graph-associahedra
gives the following recurrence behavior of $F_q(\Gamma)$.

\begin{proposition}\label{recgraph}
The quasisymmetric function $F_q(\Gamma)$ associated to a simple
graph $\Gamma$ satisfies

\begin{item}
\item[(1)] $F_q(\bullet)=M_{(1)}$ for the one-vertex graph,

\item[(2)] If $\Gamma=\Gamma_1\sqcup \Gamma_2$ then
$F_q(\Gamma)=F_q(\Gamma_1)F_q(\Gamma_2)$,

\item[(3)] If $\Gamma$ is connected then
$F_q(\Gamma)=\sum_{I\varsubsetneq
V}q^{|V|-|I|-1}(F_q(\Gamma\mid_I))_{|V|-|I|}.$
\end{item}
\end{proposition}

\noindent The Proposition \ref{recgraph} includes the recurrence
formula for the special case $q=0$ obtained in \cite[Theorem
7.4]{GV1}

\begin{corollary}\label{deletion} The quasisymmetric graph
invariant $F(\Gamma)$ satisfies
\begin{item}
\item[(1)] $F(\bullet)=M_{(1)}$ for the one-vertex graph,

\item[(2)] If $\Gamma=\Gamma_1\sqcup \Gamma_2$ then
$F(\Gamma)=F(\Gamma_1)F(\Gamma_2)$,

\item[(3)] If $\Gamma$ is connected then $F(\Gamma)=\sum_{v\in
V}(F_{\Gamma\setminus v})_1,$ where $\Gamma\setminus v$ is the
induced graph on $V\setminus\{v\}$.
\end{item}
\end{corollary}

\subsection{Graphs with the same weighted enumerator $F_q(\Gamma)$}

The following nonisomorphic graphs on six vertices with the same
quasisymmetric invariant $F(\Gamma)$ were found in \cite[Example
7.5]{GV1}, see Figure \ref{same}. We obtain from Corollary
\ref{deletion}

\[F(\Gamma_1)=F(\Gamma_2)=720M_{(1,1,1,1,1,1)}+96M_{(2,1,1,1,1)}+24M_{(1,2,1,1,1)}.\]

\begin{figure}[h!]\centering
\begin{tikzpicture}[scale=.6]

\draw (-.75,0)--(0,2)--(.75,0)--(0,-2)--(-.75,0);\draw
(-.75,0)--(.75,0); \draw (-.75,0)--(-1.5,.5); \draw
(.75,0)--(1.5,.5); \draw
(0,-2)--(-1.5,.5)--(0,2)--(1.5,.5)--(0,-2);
\draw[fill](-.75,0)circle(2pt);\draw[fill](.75,0)circle(2pt);
\draw[fill](0,2)circle(2pt);\draw[fill](0,-2)circle(2pt);\draw[fill](-1.5,.5)circle(2pt);
\draw[fill](1.5,.5)circle(2pt);\node[right] at
(1.5,-2){$\Gamma_1$};

\draw (6.25,0)--(7,2)--(7.75,0);\draw (7,-2)--(6.25,0); \draw
(6.25,0)--(7.75,0); \draw (6.25,0)--(5.5,.5); \draw
(7.75,0)--(8.5,.5); \draw (7,-2)--(5.5,.5)--(7,2);\draw
(8.5,.5)--(7,2);\draw (8.5,.5)--(7,-2);
\draw[dashed](7.75,0)--(5.5,.5);
\draw[fill](6.25,0)circle(2pt);\draw[fill](7.75,0)circle(2pt);
\draw[fill](7,2)circle(2pt);\draw[fill](7,-2)circle(2pt);\draw[fill](5.5,.5)circle(2pt);
\draw[fill](8.5,.5)circle(2pt);\node[right] at
(8.5,-2){$\Gamma_2$};

\end{tikzpicture}
\caption{Graphs with $F(\Gamma_1)=F(\Gamma_2)$} \label{same}
\end{figure}
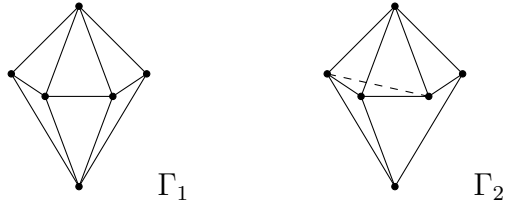

By applying Proposition \ref{recgraph} we obtain that
$F_q(\Gamma)$ of these graphs are also equal. In Table \ref{table}
are given coefficients of $F_q(\Gamma_1)=F_q(\Gamma_2)$ by the
monomials $q^{k}, k=0,1,2,3,4,5$.

\begin{table}\centering
\begin{tabular}{|c|c|}
\hline
  $q^{5}$ & $M_{(6)}$ \\
  \hline
  $q^{4}$ &
  $6M_{(1,5)}+6M_{(5,1)}+11M_{(2,4)}+15M_{(4,2)}+18M_{(3,3)}$\\
  \hline
  $q^{3}$ &
  $30M_{(1,1,4)}+30M_{(1,4,1)}+30M_{(4,1,1)}+66M_{(2,2,2)}$\\
   & $+4M_{(2,4)}+2M_{(3,3)}+56M_{(1,2,3)}+60M_{(1,3,2)}+$\\
   & $+44M_{(2,1,3)}+44M_{(2,3,1)}+54M_{(3,1,2)}+54M_{(3,2,1)}$
\\
\hline
  $q^{2}$ &
  $108M_{(3,1,1,1)}+120M_{(1,3,1,1)}+120M_{(1,1,3,1)}+$\\
  & $+120M_{(1,1,1,3)}+180M_{(1,1,2,2)}+168M_{(1,2,1,2)}+$\\
  & $+168M_{(1,2,2,1)}+132M_{(2,1,1,2)}+132M_{(2,1,2,1)}+$\\
  & $+132M_{(2,2,1,1)}+4M_{(1,2,3)}+16M_{(2,1,3)}+$\\
& $+16M_{(2,3,1)}+6M_{(3,1,2)}+6M_{(3,2,1)}+24M_{(2,2,2)}$\\
\hline
  $q^{1}$ &
  $360M_{(1,1,1,1,2)}+360M_{(1,1,1,2,1)}+360M_{(1,1,2,1,1)}+$\\
  & $+336M_{(1,2,1,1,1,)}+264M_{(2,1,1,1,1)}+12M_{(1,2,1,2)}+$\\
  & $+12M_{(1,2,2,1)}+48M_{(2,1,1,2)}+48M_{(2,1,2,1)}+$\\
  & $+48M_{(2,2,1,1)}+12M_{(3,1,1,1)}$ \\
\hline
  $q^{0}$ & $720M_{(1,1,1,1,1)}+96M_{(2,1,1,1,1)}+24M_{(1,2,1,1,1)}$ \\
  \hline
\end{tabular}
\caption{$F_q(\Gamma_1)=F_q(\Gamma_2)$}\label{table}
\end{table}

Theorem \ref{general} then implies that $f$-polynomials of
corresponding graph-associahedra are equal. We found
$$f(P_{\Gamma_1},q)=f(P_{\Gamma_2},q)=q^{5}+56q^{4}+462q^{3}+1308q^{2}+1500q+600.$$

Nevertheless the polytopes $P_{\Gamma_i}, i=1,2$ are not
combinatorially equivalent. We can show this by looking at the
$1$-skeletons of dual polytopes
$G_i=((P_{\Gamma_i})^{\ast})^{(1)}, i=1,2.$ These graphs have
different sequences of vertex degrees, see Tables \ref{degrees1}
and \ref{degrees2}. Here $d$ stands for the vertex degree and $n$
for the number of vertices with this degree. Recall that dual
polytopes of nestohedra are represented by nested set complexes,
see \cite[Theorem 6.5]{PRW}. Therefore $G_i, i=1,2$ are graphs
whose vertices correspond to connected induced subgraphs of
$\Gamma_i, i=1,2$ and edges are given by pairs of connected
induced subgraphs $\{H_1,H_2\}$ such that either one is contained
into another or their union is not connected.

\begin{table}[h!]\centering
\begin{tabular}{|c|c|c|c|c|c|c|c|c|c|c|c|}
\hline
  d & 30 & 29 & 28 & 26 & 25 & 16 & 15 & 14 & 13 & 12 & 11 \\
  \hline
  n & 4 & 2 & 2 & 2 & 2 & 9 & 9 & 3 & 4 & 6 & 13\\
  \hline
\end{tabular}
\caption{Vertex degrees of $G_1$}\label{degrees1}
\end{table}

\begin{table}[h!]\centering
\begin{tabular}{|c|c|c|c|c|c|c|c|c|c|c|}
\hline
  d & 30 & 29 & 28 & 26 & 16 & 15 & 14 & 13 & 12 & 11 \\
  \hline
  n & 2 & 4 & 2 & 4 & 10 & 6 & 4 & 6 & 6 & 12\\
  \hline
\end{tabular}
\caption{Vertex degrees of $G_2$}\label{degrees2}
\end{table}

\begin{remark}

By using MathLab programm we found exactly three pairs of
six-vertex graphs with the same quasisymmetric function
$F(\Gamma)$. If we label adequately the vertices then the sets of
edges of graphs given in Figure \ref{same} are
$E(\Gamma_1)=\{12,13,14,15,23,34,45,26,36,46,56\}$ and
$E(\Gamma_2)=\{12,13,14,15,23,24,34,45,26,36,56\}$. Two other
pairs are obtained by deleting edges $14$ and $14, 26$ from
$\Gamma_1$ and $\Gamma_2$ respectively. Similarly as for the pair
$\Gamma_1$ and $\Gamma_2$ it can be shown that the other two pairs
have equal $q$-analogs $F_q(\Gamma)$ and henceforth equal
$f$-polynomials of corresponding $P_\Gamma$, which are also
combinatorially nonequivalent polytopes.
\end{remark}

\section{Matroid base polytope}

A standard reference monograph for matroid theory is \cite{O}. The
most adequate for our purposes is a definition of a matroid as a
collection of bases. A matroid $M=([n],\mathcal{B})$ on the ground
set $[n]$ is a nonempty collection $\mathcal{B}$ of subsets of
$[n]$ such that the exchange property is satisfied:

\begin{itemize}
\item For each $B_1, B_2\in\mathcal{B}$ and $i\in B_1\setminus
B_2$ there is $j\in B_2\setminus B_1$ such that
$B_1\setminus\{i\}\cup\{j\}\in\mathcal{B}$.
\end{itemize}
The members of $\mathcal{B}$ are called bases of $M$ and we denote
the collection of bases of $M$ by $\mathcal{B}(M)$. All bases have
the same number of elements which is called the rank $r(M)$ of the
matroid $M$.

\begin{definition} The matroid base polytope is a convex polytope
$$P_M=\mathrm{Conv}\left\{e_B \mid B\in\mathcal{B}(M)\right\},$$
where $e_B=\displaystyle\sum_{i\in B}e_i$ and $e_i, i=1,\ldots,n$
are the standard basis vectors in $\mathbb{R}^n.$
\end{definition}
The polytope $P_M$ is contained in the hypersimplex
$$\Delta_r^{n}=[0,1]^{n}\cap\{x_1+x_2+\cdots+x_n=r(M)\}.$$
The following characterization of matroids in terms of their base
polytopes shows that matroid base polytopes are generalized
permutohedra.

\begin{theorem}[\cite{GelSer}, Section 2.2, Theorem 1 and \cite{GGMS}, Theorem 4.1]\label{characterization}
Let $\mathcal{S}$ be a collection of $r$-subsets of $[n]$ and
$P_\mathcal{S}$ be a convex polytope in $\mathbb{R}^n$
$$P_\mathcal{S}=\mathrm{Conv}\left\{e_S\mid S\in
\mathcal{S}\right\},$$ where $e_S=\sum_{i\in S}e_i$. Then
$\mathcal{S}$ is the collection of bases of a matroid if and only
if every edge of $P_\mathcal{S}$ is a translate of the vector
$e_i-e_j,$ for some $i,j\in[n].$
\end{theorem}

\begin{definition} For a matroid $M=([n],\mathcal{B})$ the elements
$i, j\in[n]$ are equivalent if there exist bases $B_1$ and $B_2$
with $i\in B_1$ and $B_2=(B_1\setminus\{i\})\cup\{j\}$. The
equivalence classes are called \emph{connected components} of $M$.
The matroid $M$ is \emph{connected} if it has only one connected
component.
\end{definition}
Denote by $c(M)$ the number of connected components of a matroid
$M$.

\begin{proposition}[\cite{FS}, Proposition 2.4]\label{dimmbp} The dimension of the matroid base polytope $P_M$
corresponding to a matroid $M$ on $[n]$ is determined by
 $$\dim(P_M)=n-c(M).$$
\end{proposition}

\subsection{Flags of matroid base polytope}

For a vector
$\omega=(\omega_1,\omega_2,\ldots,\omega_n)\in\mathbb{Z}^n_+$ we
regard the function $\omega^{\ast}$ on the matroid base polytope
$P_M$ as a weight function on a matroid $M$, so that the
$\omega$-weight of a basis
$B=\{i_1,\ldots,i_{r(M)}\}\in\mathcal{B}(M)$ is given by
$$\omega^{\ast}(B)=\langle\omega,e_B\rangle=\omega_{i_1}+\cdots+\omega_{i_{r(M)}}.$$
Let $\mathcal{B}_{\omega}\subset\mathcal{B}(M)$ be the collection
of bases of $M$ having maximum $\omega$-weight. The collection
$\mathcal{B}_\omega$ consists of all bases obtained by greedy
algorithm. The polytope $P_{\mathcal{B}_\omega}$ is the face of
$P_M$ on which $\omega^{\ast}$ is maximized. Theorem
\ref{characterization} implies that $\mathcal{B}_\omega$ is the
collection of bases of a matroid which we denote by $M_\omega$.

Recall that $\omega$ determines a unique flag of subsets

\begin{equation}\label{flag}
\mathcal{F}_\omega:\emptyset=:F_0\subset F_1\subset\ldots\subset
F_{k}\subset F_{k+1}:=[n]
\end{equation}
for which $\omega$ is constant on $F_i\setminus F_{i-1},
i=1,\ldots,k+1$ and satisfies $\omega|_{F_{i+1}\setminus
F_i}<\omega|_{F_i\setminus F_{i-1}}, i=1,\ldots k$. The following
proposition shows that $M_\omega$ depends only on the flag
$\mathcal{F}_\omega$. By $r(A)=r(M\mid_A)$, where $M\mid_A$ is the
restriction of $M$ to $A\subset[n]$, is defined the rank function
of $M$.

\begin{proposition}[\cite{AK}, Proposition 1] Let
$\omega\in\mathbb{Z}^{n}_+$ and $\mathcal{F}=\mathcal{F}_\omega$
is the corresponding flag given by $(\ref{flag})$. A base
$B\in\mathcal{B}(M)$ of a matroid $M$ has the maximum
$\omega$-weight, i.e. $B\in B_\omega$ if and only if
$|B\cap(F_i\setminus F_{i-1})|=r(F_i)-r(F_{i-1}), i=1,\ldots,k+1$.
\end{proposition}

We say that a flag $\mathcal{F}$ is maximized at a base
$B\in\mathcal{B}(M)$ if and only if $B\in B_\omega$ and
$\mathcal{F}=\mathcal{F}_\omega$ for some
$\omega\in\mathbb{Z}^{n}_+$. Denote by $M/\mathcal{F}$ the matroid
$M_\omega$ corresponding to a flag
$\mathcal{F}=\mathcal{F}_\omega$. We determine the map
$\pi_{P_M}:L(Pe^{n-1})\rightarrow L(P_M)$ associated to the
matroid base polytope by definition \ref{lattice}, namely
$$\pi_{P_M}(\mathcal{F})=P_{M/\mathcal{F}}.$$ The following
proposition describes the matroid $M/\mathcal{F}$ as a sum of its
minors.

\begin{proposition}[\cite{AK}, Proposition 2]\label{decomposition}
If $\mathcal{F}=\{\emptyset=:F_0\subset\ldots\subset
F_{k+1}:=[n]\}$, then
$$M/\mathcal{F}=\bigoplus_{i=1}^{k+1}\left(M|_{F_{i}}\right)/F_{i-1}.$$
\end{proposition}

Since for the direct sum of matroids we have $P_{M_1\oplus
M_2}=P_{M_1}\times P_{M_2}$, by Proposition \ref{dimmbp} we obtain

\begin{corollary}\label{M-rank}
The $P_M$-rank of a flag $\mathcal{F}$ is given by
$$\mathrm{rk}_{P_M}(\mathcal{F})=\dim(P_{M/\mathcal{F}})=n-\displaystyle\sum_{i=1}^{k}c\left((M|F_i)/F_{i-1}\right).$$
\end{corollary}

\subsection{Quasisymmetric enumerator $F_q(P_M)$ for matroids}

The Hopf algebra of matroids $\mathcal{M}at$ was introduced by
Schmitt \cite{Sch} and more intensively studied by Crapo and
Schmitt \cite{CS1}, \cite{CS2}.

As a vector space $\mathcal{M}at$ is linearly spanned over the
field $\mathbf{k}$ by isomorphism classes of finite matroids
$[M]$. The direct sum induces a product, while a coproduct is
determined by restrictions and contractions
\[[M_1]\cdot[M_2]=[M_1\oplus
M_2]\;\;\;\;\;\mbox{and}\;\;\;\;\;\Delta[M]=\displaystyle\sum_{A\subseteq
E}[M|_A]\otimes[M/A].
\]
The Hopf algebra $\mathcal{M}at$ is graded
$\mathcal{M}at=\bigoplus_{n\geq0}\mathcal{M}at_n,$ where
$\mathcal{M}at_n$ denotes the subspace spanned by elements $[M]$
for which the ground set has cardinality $n$. The Hopf algebra
$\mathcal{M}at$ is graded, connected, commutative, but
non-cocommutative. The unit is the class of a unique matroid on
the empty set $[M_{\emptyset}]$.

The character $c:\mathcal{M}at\rightarrow\bf{k}$ defined by
$$c([M])=\left\{\begin{array}{cc} 1,& M \ \mbox{is a direct sum of loops and isthmuses},\\
0,& \mbox{otherwise}\end{array}\right.$$ turns $\mathcal{M}at$
into a combinatorial Hopf algebra considered by Billera, Jia and
Reiner in \cite{BJR}. By Theorem \ref{fundamental} there is a
unique morphism $\Psi:(\mathcal{M}at,c)\rightarrow
(QSym,\zeta_{\mathcal{Q}})$ of combinatorial Hopf algebras given
in monomial bases of quasisymmetric functions by
$$\Psi([M])=\displaystyle\sum_{\alpha\models n}c_\alpha([M])M_{\alpha}.$$

\begin{proposition}[\cite{BJR}, Proposition 3.3] The coefficient $c_\alpha([M])$ is the
number of flags $\mathcal{F}:\emptyset=F_0\subset\cdots\subset
F_{k+1}=[n]$ having $\mathrm{type}(\mathcal{F})=\alpha$ and for
which each subquotient $(M|_{F_i})/F_{i-1}$ is a direct sum of
loops and isthmuses.
\end{proposition}

We extend the basic field $\bf k$ into the field of rational
functions ${\bf k}(q)$ and define the character
$c_q:\mathcal{M}at\rightarrow{\bf k}(q)$ with
$$c_q([M])=q^{\mathrm{rk}(M)},$$ where $\mathrm{rk}(M)=n-c(M)$.
Let
$\Psi_q:(\mathcal{M}at,c_q)\rightarrow(\mathcal{Q}Sym,\zeta_{\mathcal{Q}})$
be a unique morphism of combinatorial Hopf algebras over ${\bf
k}(q)$.

\begin{theorem}\label{mat}

For a matroid $M$ the quasisymmetric enumerator function
$F_q(P_M)$ associated to the matroid base polytope $P_M$ coincides
with the value of a universal morphism  $$F_q(P_B)=\Psi_q([M]).$$

\end{theorem}

\begin{proof}
The morphism $\Psi_q$ is given by
$$\Psi_q([M])=\displaystyle\sum_{\alpha\models n}c_{q,\alpha}([M])M_\alpha,$$
where
$$c_{q,\alpha}([M])=\displaystyle\sum_{\mathcal{F}:\mathrm{type}(\mathcal{F})=\alpha}
\displaystyle\prod_{i=1}^{k}q^{\mathrm{rk}\left((M|_{F_i})/F_{i-1}\right)}.$$
By Corollary \ref{M-rank} we have
$$c_{q,\alpha}([M])=\displaystyle\sum_{\mathcal{F}:\mathrm{type}(\mathcal{F})=\alpha}q^{\mathrm{rk}_{P_M}(\mathcal{F})}.$$
Consequently,
$$\Psi_q([M])=\displaystyle\sum_{\mathcal{F}}
q^{\mathrm{rk}_{P_M}(\mathcal{F})}M_{\mathrm{type}(\mathcal{F})},$$
where the sum is over all flags of $[n]$. This is exactly the form
of $F_q(P_M)$ given by identity $(\ref{expansion})$.
\end{proof}

As a consequence of Theorem \ref{mat} and Theorem \ref{general} we
obtain an algebraic description of the $f$-polynomials of matroid
base polytopes.

\begin{theorem}\label{algebraic} The $f$-polynomial of a matroid base polytope $P_M$ is given by
$$f(P_M,q)=(-1)^n\mathbf{ps}(\Psi_{-q}([M]))(-1).$$
\end{theorem}

\subsection{The uniform matroid base polytope}
\bigskip

We use theorem \ref{algebraic} to calculate $f$-polynomials of
uniform matroid base polytopes $P_{U_{n,r}}$. The uniform matroid
$U_{n,r}$ is a matroid with the set of bases
$\mathcal{B}(U_{n,r})={[n] \choose r}$ consisting of all
$r$-elements subsets of $[n]$. From definition it follows that
$P_{U_{n,r}}$ is the hypersimplex
$$P_{U_{n,r}}=\Delta_r^{n}.$$ Note that
$\mathcal{B}(U_{n,0})=\{\emptyset\}$ and
$\mathcal{B}(U_{n,n})=\{[n]\},$ so $P_{U_{n,0}}=\mathbf{0}$ and
$P_{U_{n,n}}=\mathbf{1}$ are single points, where
$\mathbf{0}=(0,\ldots,0)$ and $\mathbf{1}=(1,\ldots,1)$.

Consider the uniform matroid $U_{n,r}$ and assume that $0<r<n$.
Let $\mathcal{F}:\emptyset=F_0\subset F_1\subset\ldots\subset
F_{k+1}=[n]$ be a flag with
$\mathrm{type}(\mathcal{F})=(\alpha_1,\ldots,\alpha_i,\alpha_{i+1},\ldots,\alpha_{k}).$
Define
$i_0=i_0(\mathcal{F})=\min\{i\mid\alpha_1+\cdots+\alpha_i\geq r\}$
and partial sums $r'=\alpha_1+\cdots+\alpha_{i_0-1},
r''=r'+\alpha_{i_0}$. By proposition \ref{decomposition} we find

$$M/\mathcal{F}=U_{r',r'}\oplus U_{\alpha_{i_0},r-r'}\oplus
U_{n-r'',0}.$$ It shows that faces of the uniform matroid base
polytopes are uniform matroid base polytopes as well, the fact
that is obvious from the description by hypersimplices. We have
$P_{M/\mathcal{F}}\cong P_{U_{\alpha_{i_0},r-r'}}$, which implies
the following formula for the $P_M$-rank, where $M=U_{n,r}$

\begin{equation}\label{uniformrank}
\mathrm{rk}_{P_M}(\mathcal{F})=\dim(P_{M/\mathcal{F}})=
\left\{\begin{array}{cc} 0,& r=r'',\\
\alpha_{i_0}-1,& r<r''\end{array}\right..
\end{equation}

Let $\circ$ be the concatenation product defined on monomial bases
by

$$M_\alpha\circ M_\beta=M_{\alpha\beta}$$
and linearly extended to the algebra of quasisymmetric functions.
We obtain from $(\ref{uniformrank})$

$$F_q(P_{U_{n,r}})={n \choose r}(M_1)^r\circ(M_1)^{n-r}+$$

$$+\displaystyle\sum_{0\leq r^{\prime}<r<r^{\prime}+\lambda\leq
n}{n \choose r^{\prime}}{n-r^{\prime} \choose
\lambda}q^{\lambda-1}(M_1)^{r^{\prime}}\circ
M_{\lambda}\circ(M_1)^{n-r^{\prime}-\lambda}.$$ The principal
specilization evaluated at $-1$ gives
$$f(P_{U_{n,r}},q)={n \choose r}+\displaystyle\sum_{0\leq r^{\prime}<r<r^{\prime}+\lambda\leq n}
{n \choose r^{\prime}}{n-r^{\prime}\choose
\lambda}q^{\lambda-1}.$$ It determines the $f$-vector of
$P_{U_{n,r}}$ in terms of trinomial coefficients
$$f_0={n \choose r} \ \ \mbox{and} \ \ f_{k-1}=\displaystyle\sum_{0\leq r^{\prime}<r<r^{\prime}+k\leq
n} {n \choose r', k, n-r'-k} \ \ \mbox{for} \ \ k=2,\ldots,n.$$

\begin{example} The polytope $P_{U_{4,2}}=\Delta_2^{4}$ is an
octahedron. We obtain
$$F_q(P_{U_{4,2}})=6\left(M_{(2,2)}+2M_{(1,1,2)}+2M_{(2,1,1)}+4M_{(1,1,1,1)}\right)q^0+12M_{(1,2,1)}q
+$$ $$+4\left(M_{(1,3)}+M_{(3,1)}\right)q^2+M_{(4)}q^3,$$ which
yields to the expected expression
$f(P_{U_{4,2}},q)=6+12q+8q^2+q^3.$
\end{example}

\section{Properties of $F_q(M)$}

In the following we will write $F_q(M)$ instead of $F_q(P_M)$. We
examine some algebraic properties of the quasisymmetric enumerator
$F_q(M)$. First we give an example which shows that $F_q(M)$
contains more information about matroids than its specialization
at $q=0$. The following example of matroids with the same
quasisymmetric invariant $F(M)$ is borrowed from \cite[Example
8.1]{BJR}.

\begin{figure}[h!]\centering
\begin{tikzpicture}[scale=1]

\draw (-1,.5)--(0,.5)--(1,.5);  \draw (-1,-.5)--(0,-.5)--(1,-.5);
\draw[fill](-1,.5)circle(1pt);\draw[fill](0,.5)circle(1pt);\draw[fill](1,.5)circle(1pt);
\draw[fill](-1,-.5)circle(1pt);\draw[fill](0,-.5)circle(1pt);\draw[fill](1,-.5)circle(1pt);

\draw (4,1)--(4,0)--(4,-1)--(5,-1)--(6,-1);
\draw[fill](4,1)circle(1pt);\draw[fill](4,0)circle(1pt);\draw[fill](4,-1)circle(1pt);
\draw[fill](5,-1)circle(1pt);\draw[fill](6,-1)circle(1pt);\draw[fill](5.5,.5)circle(1pt);

\node[right] at (1,-1) {$M_1$}; \node[right] at (6,-1) {$M_2$};

\node[below]at (-1,.5) {\tiny $1$};\node[below]at (0,.5) {\tiny
$2$};\node[below]at (1,.5) {\tiny $3$};\node[below]at (-1,-.5)
{\tiny $4$};\node[below]at (0,-.5) {\tiny $5$};\node[below]at
(1,-.5) {\tiny $6$};

\node[left]at (4,1) {\tiny $1$};\node[left]at (4,0) {\tiny
$2$};\node[below left]at (4,-1) {\tiny $3$};\node[below]at (5,-1)
{\tiny $4$};\node[below]at (6,-1) {\tiny $5$};\node[above right]at
(5.5,.5) {\tiny $6$};

\end{tikzpicture}
\caption{Matroids with $F(M_1)=F(M_2)$}\label{matroids}
\end{figure}
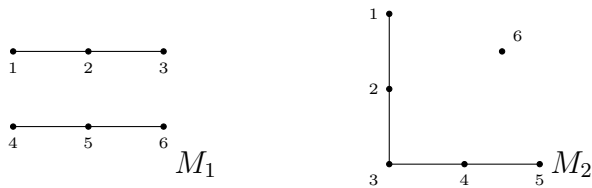

\begin{example}
Let $M_1$ and $M_2$ be affine matroids depicted on Figure
\ref{matroids}. $M_1$ and $M_2$ are the rank 3 matroids on the set
$[6]$ having every triple but $\{1,2,3\}, \{4,5,6\}$ and
$\{1,2,3\}, \{3,4,5\}$ as bases, respectively. Consider flags
$\emptyset\subset\{i\}\subset[6]\setminus\{j\}\subset[6], 1\leq
i\neq j\leq 6$ of the type $(1,4,1)$. All minors
$M_1|_{[6]\setminus\{j\}}/\{i\}$ are connected, so we have the
summand $30q^{3}M_{(1,4,1)}$ in $F_q(M_1)$. On the other hand the
flag $\emptyset\subset\{3\}\subset\{1,2,3,4,5\}\subset[6]$
produces a two components minor of $M_2$ which contributes with
$q^{2}M_{(1,4,1)}$ in $F_q(M_2)$. It shows that $F_q(M_1)\neq
F_q(M_2)$.
\end{example}

\subsection{Matroid duality}

For a matroid $M$ on $[n]$ with the collection of bases
$\mathcal{B}(M)$ the dual matroid $M^{\ast}$ is defined by the set
of bases $\mathcal{B}(M^{\ast})=\{[n]\setminus B\mid
B\in\mathcal{B}(M)\}.$ The affine transformation
$\mathrm{aff}:[0,1]^{n}\rightarrow[0,1]^{n}$ of the $n$-cube
$\mathrm{aff}(x)=\mathbf{1}-x, x\in[0,1]^{n}$, where
$\mathbf{1}=(1,\ldots,1)$ determines an affine isomorphism between
$P_M$ and $P_{M^{\ast}}$. The behavior of the quasisymmetric
invariant $F_q(M)$ for $q=0$ under matroid duality is determined
by \cite[Proposition 4.1]{BJR}. We generalize this to the
$q$-analog $F_q(M)$.

\begin{lemma}\label{aff}
For a matroid $M$ and a flag $\mathcal{F}$ on its ground set $[n]$
it holds

$$\mathrm{aff}(P_{M/\mathcal{F}})=P_{M^{\ast}/\mathcal{F}^{op}}.$$
\end{lemma}
\begin{proof}
The statement follows from the fact that the flag $\mathcal{F}$ is
maximized at a base $B\in\mathcal{B}(M)$ if and only if
$\mathcal{F}^{op}$ is maximized at the cobase $[n]\setminus
B\in\mathcal{B}(M^{\ast})$.
\end{proof}

\begin{proposition}
The quasisymmetric function $F_q(M)$ satisfies
$F_q(M)=\sum_{\alpha\models n}p_\alpha(q)M_\alpha$ if and only if
$F_q(M^{\ast})=\sum_{\alpha\models
n}p_\alpha(q)M_{\mathrm{rev}(\alpha)}.$
\end{proposition}

\begin{proof}
We use the expansion $(\ref{poQ})$
$$F_q(M^{\ast})=\sum_{G\in
L(P_{M^{\ast}})}q^{\mathrm{dim}(G)}\sum_{\mathcal{F}:P_{M^{\ast}(\mathcal{F})}=G}M_{\mathrm{type}(\mathcal{F})}.$$
Denote by $G^{\ast}=\mathrm{aff}(G)$. Then by lemma \ref{aff} and
the fact $\mathrm{dim}(G)=\mathrm{dim}(G^{\ast})$ we can write
$$F_q(M^{\ast})=\sum_{G^{\ast}\in
L(P_{M})}q^{\mathrm{dim}(G^{\ast})}\sum_{\mathcal{F}:P_{M(\mathcal{F}^{op})}=G^{\ast}}M_{\mathrm{type}(\mathcal{F}^{op})}.$$
The proposition follows from $(\ref{rev})$.
\end{proof}

\end{document}